\documentclass[12pt,a4paper]{amsart}
\usepackage{amssymb}
\usepackage{hyperref}

\usepackage{graphicx,amssymb,amsfonts,epsfig,amsthm,a4,amsmath,url}
\usepackage[latin1]{inputenc}

\usepackage{amsmath, amsthm, amssymb, verbatim}
\usepackage{graphicx,amssymb,amsfonts,epsfig,amsthm,a4,amsmath,url,graphicx,enumerate}
\usepackage[latin1]{inputenc}
\usepackage{tikz}
\usepackage{textcomp}
\usetikzlibrary{matrix}
\usepackage{svg}
\usepackage{float}

\newtheorem{thm}{Theorem}[section]
\newtheorem{cor}[thm]{Corollary}
\newtheorem{lem}[thm]{Lemma}

\newtheorem{prop}[thm]{Proposition}
\theoremstyle{definition}
\newtheorem{defn}[thm]{Definition}

\theoremstyle{remark}
\newtheorem{rem}[thm]{Remark}

\newtheorem{ex}[thm]{Example}

\numberwithin{equation}{section}

\newcommand\restr[2]{{
  \left.\kern-\nulldelimiterspace 
  #1 
  \littletaller 
  \right|_{#2} 
  }}

\newcommand{\littletaller}{\mathchoice{\vphantom{\big|}}{}{}{}}

\begin{document}
\title[On super-rigidity of Gromov's random monster group]{On super-rigidity of Gromov's random monster group}
\author[Das]{Kajal Das}
\address{Statistics and Mathematics Unit, Indian Statistical Institute\\  203 Barrackpore Trunk Road\\Kolkata 700 108, India}
\email{kdas.math@gmail.com}

\maketitle
\textbf{Abstract:} In this article, we show super-rigidity of Gromov's random monster group. We prove that any morphism $\phi_\alpha$ from Gromov's random monster group $\Gamma_\alpha$ to the group $G$ has finite image for almost all $\alpha$, where $G$ is any of the following types of groups: mapping class group $MCG(S_{g,b})$, braid group $B_n$,  outer automorphism group of a free group $Out(F_N)$,  automorphism group of a free group $Aut(F_N)$,  hierarchically hyperbolic group, a-$L^p$-menable group or K-amenable group. We introduce another property called hereditary super-rigidity and prove that $\Gamma_\alpha$ has hereditary super-rigidity with respect to an a-$L^p$-menable group or a K-amenable group. We also establish a stability theorem for the groups with respect to which $\Gamma_\alpha$ has super-rigidity and hereditary super-rigidity. 

\vspace{5mm}


\vspace{5mm}

\textbf{Key terms:}:  Gromov's random monster group, Gromov hyperbolic space, elementary and non-elementary action, mapping class group, braid group, automorphism and outer automorphism group of a free group, acylindrically hyperbolic group, hierarchically hyperbolic group, a-$L^p$-menable group, K-amenable group.

\section{Introduction}

In the article \cite{Gro03}, Gromov first introduced `random monster group' as a random model of finitely generated infinitely presented groups.  We briefly recall the construction of Gromov's random monster group. Given three integers $d\geq 3$, $k\geq 2$ and $j\geq 1$ and an infinite sequence of large girth finite $d$-regular connected graphs $\{\Omega_n\}$, we label independently at random each edge of every $\Omega_n$
with a word of length $j$ in the free group on $k$-generators $F_k$. We define a group $\Gamma_\alpha$ (corresponding to a fixed labelling $\alpha$) as the quotient of $F_k$ by the normal closure of 
the set of words corresponding to closed loops of $\{\Omega_n\}$. The collection of groups $\{\Gamma_\alpha\}$ can be given a natural probability measure. The main interest of this construction is that for appropriate parameters $d,k,j$, if one chooses $\{\Omega_n\}$ to be a suitable sequence of expander graphs, the group $\Gamma_\alpha$ almost surely does not coarsely embed into a Hilbert space, because the sequence $\{\Omega_n\}$ ``weakly'' embeds in the Cayley graph of $\Gamma_\alpha$. Also, it provides a counterexample to the Baum-Connes conjecture for groups with coefficients in commutative $C^*$-algebra \cite{HLS02}. For more information on Gromov's random monster group, we refer the reader 
to the article \cite{AD08}.

However, the first rigidity property (or fixed-point property) of Gromov's random monster group can be seen in terms of Property (T) of this group (see \cite{Gro03}, \cite{Sil03}). We recall that a discrete group has \textit{Property (T)} if and only if every isometric affine action of the discrete group on a Hilbert space Y has a fixed point. Later, this result has been generalized in \cite{NS11} for $p$-uniformly convex metric spaces and for a high girth $p$-expanders $\{\Omega_n\}$. The metric spaces include $L^p$-space ($1< p<\infty$), CAT(0) spaces etc. It is well known that if a discrete group has Property (T), then it satisfes Serre's Property (FA). We racall that a group satisfies \textit{Property (FA)} if every isometric action on a simplicial tree fixes a point (\cite{Ser03}). Property (FA) can be broadly generalized as hyperbolically rigid . A group $G$ is said to be \textit{hyperbolically rigid} if any isometric action of the group on a Gromov-hyperbolic space is elementary (see Subsection 2.2 for the definition of elementary action). This kind of rigidity result has been obtained for Gromov's random monster group by Gruber-Sisto-Tessera in \cite{GST20} under a different condition on $\{\Omega_n\}$.  They assume that the diameter-by-girth ratio of the graphs $\{\Omega_n\}$ are uniformly bounded instead of the condition that $\{\Omega_n\}$ is a sequence of expanders.   The details of these results can be obtained in Subsection 2.4. 

In this article, we study the following type of question: Let $\phi_\alpha:\Gamma_\alpha\rightarrow G $ be a group homomorphism for all $\alpha$, where $G$ is a countable discrete group. For which groups $G$, the image of $\phi_\alpha$ is finite for a.e. $\alpha$? If a group $G$ satisfies the above property, we say that $\{\Gamma_\alpha\}$ has \textit{super-rigidity} with respect to $G$. If $G$ is a linear group (i.e. a subgroup of $GL_m(K)$, where $K$ is a field), we say that $\{\Gamma_\alpha\}$ has \textit{linear super-rigidity} with respect to $G$; for other groups we say that it has \textit{non-linear super-rigidity}. In this article, we prove such super-rigidity results with respect to mapping class groups, braid groups, $Out(F_N)$(outer automorphism group of free group with $N$ generators), $Aut(F_N)$ (automorphism group of free group with $N$ generators) hierarchically hyperbolic groups, a-$L^p$-menable groups . For acylindrically hyperbolic group $G$, we prove a weaker result. We prove that $\Gamma_\alpha$ has elliptic image in $G$ for a.e. $\alpha$. In our results, we assume that the graphs $\{\Omega_n\}$ is a sequence of $d$-regular expander graphs, the diameter-by-girth ratio of the graphs are bounded by $C$, i.e., $diam(\Omega_n)\leq C girth (\Omega_n)$ for all $n\in\mathbb{N}$ and $girth(\Omega_n)\rightarrow\infty$ as $n\rightarrow\infty$. For the case of a-$L^P$-menability group, we assume that $\{\Omega_n\}_{n\in\mathbb{N}}$ is a $p$-expander with respect to $L^p$-space for $1< p<\infty$ and $girth(\Omega_n)\geq C\hspace*{1mm} log(|\Omega_n|)$ for some constant $C$ and for all $n\in\mathbb{N}$ . Now, we describe our main results. The notations used in these results can be found in Section 2.

\begin{thm}\label{mcg}
For every $j\geq 1$ and for any morphism $\Gamma_\alpha\rightarrow MCG(S_{g,b})$, $\Gamma_\alpha$ has finite image in $MCG(S_{g,b})$ for almost every $\alpha\in \mathcal{A}(\Omega, T^j)$.
\end{thm}

\begin{cor}\label{braid}

For every $j\geq 1$ and for any morphism $\Gamma_\alpha\rightarrow B_n$, $\Gamma_\alpha$ has finite image in $B_n$ for almost every $\alpha\in \mathcal{A}(\Omega, T^j)$.

\end{cor}

\begin{thm}\label{outfn}
For every $j\geq 1$ and for any morphism $\Gamma_\alpha\rightarrow Out(F_n)$, $\Gamma_\alpha$ has finite image in $Out(F_n)$ for almost every $\alpha\in \mathcal{A}(\Omega, T^j)$.
\end{thm}

\begin{cor}\label{autfn}
For every $j\geq 1$ and for any morphism $\Gamma_\alpha\rightarrow Aut(F_n)$, $\Gamma_\alpha$ has finite image in $Aut(F_n)$ for almost every $\alpha\in \mathcal{A}(\Omega, T^j)$.
\end{cor}

\begin{thm}\label{acylin}
Let $G$ be an acylindrically hyperbolic group. For every $j\geq 1$ and for any morphism $\Gamma_\alpha\rightarrow G$, $\Gamma_\alpha$ has elliptic image in $G$  for almost every $\alpha\in \mathcal{A}(\Omega, T^j)$.
\end{thm}

\begin{thm}\label{hhg}
Let $G$ be a hierarchically hyperbolic group. For every $j\geq 1$ and for any morphism $\Gamma_\alpha\rightarrow G$, $\Gamma_\alpha$ has finite image in $G$ for almost every $\alpha\in \mathcal{A}(\Omega, T^j)$.
\end{thm}

\begin{thm}\label{a-lp-menability}
Let $G$ be an a-$L^p$-menable group. For every $j\geq 1$ and for any morphism $\Gamma_\alpha\rightarrow G$, $\Gamma_\alpha$ has finite image in $G$ for almost every $\alpha\in \mathcal{A}(\Omega, T^j)$.
\end{thm}

\begin{cor}\label{hsa-lp-menability}
$\Gamma_\alpha$ has hereditary super-rigidity with respect to any  a-$FL^p$-menable group for almost every $\alpha\in \mathcal{A}(\Omega, T^j)$, where $1< p<\infty$.

\end{cor}

\begin{thm}\label{K-amenable}
Let $G$ be a K-amenable group. For every $j\geq 1$ and for any morphism $\Gamma_\alpha\rightarrow G$, $\Gamma_\alpha$ has finite image in $G$ for almost every $\alpha\in \mathcal{A}(\Omega, T^j)$.
\end{thm}

\begin{cor}\label{hsK-amenable}
$\Gamma_\alpha$ has hereditary super-rigidity with respect to any  K-amenable group for almost every $\alpha\in \mathcal{A}(\Omega, T^j)$. 
\end{cor}

\subsection{Organization} In Section 2, we define the required concepts and notations and introduce various classes of groups. In this section, we also state the results which will be useful for proving our main theorems and corollaries.  In Section 3, we demonstrate the proofs of  our main theorems and their corollaries. In Section 4, we discuss some open questions.

\section{Some definitions, notations and useful results}

\subsection{Gromov hyperbolic space and Gromov hyperbolic group}

\begin{defn}
We say a geodesic metric space $X$ is \textit{Gromov hyperbolic}, or \textit{$\delta$-hyperbolic}, if there is a number $\delta \geq 0$ for which every geodesic triangle in $X$ satisfies the $\delta$-slim triangle condition, i.e. any side is contained in a $\delta$-neighbourhood of the other two sides.
\end{defn}

\begin{rem}
 Throughout this paper we will assume that the space $X$ is separable, proper and locally compact.
 \end{rem}
 
 \begin{defn}
 A finitely generated group $G$ is said to be \textit{Gromov hyperbolic} or \textit{hyperbolic} if its Cayley graph with respect to a finite generating set is $\delta$-hyperbolic for some $\delta\geq 0$. 
\end{defn}

\begin{rem}
A group $G$ is Gromov hyperbolic if and only if $G$ has an isometric, properly discontinuous and co-compact action on a Gromov-hyperbolic space $X$.  
\end{rem}

\subsection{Elementary and non-elementary action}
Let $G$ be a group acting isometrically on a hyperbolic metric space $X$. By $\Lambda(G)$ we denote the set of limit points of $G$ on $\partial X$, the Gromov boundary of $X$. That is, $\Lambda(G)$ is the set of accumulation points of any orbit of $G$ on $\partial X$. The possible actions of (non-cyclic) groups on hyperbolic spaces break into the following 4 classes according to $|\Lambda(G)|$ (see \cite{Gro87}, Section 8.2) :

\begin{itemize}

\item[(1)] $|\Lambda(G)|= 0$. Equivalently, $G$ has bounded orbits. In this case the action of $G$ is called \textit{elliptic}. 
 \item[(2)] $|\Lambda(G)|= 1$. Equivalently, $G$ has unbounded orbits and contains no loxodromic elements. In this case the action of $G$ is called \textit{parabolic}.
\item[(3)] $|\Lambda(G)| = 2$. Equivalently, $G$ contains a loxodromic element and any two loxodromic elements have the same limit points on $\partial X$. In this case the action of $G$ is called \textit{lineal}.
\item[(4)] $|\Lambda(G)|=\infty$. Then $G$ always contains loxodromic elements. 
In turn, this case breaks into two subcases:
\begin{itemize}
\item[(a)] $G$ fixes a point $\xi\in \partial X$. In this case $\xi$ is the common limit point of all loxodromic elements of $G$. Such an action is called \textit{quasi-parabolic}.
\item[(b)] $G$ has no fixed point in $\partial X$. $G$ contains infinitely many independent loxodromic elements. In this case the action is said to be of general type. 
\end{itemize}
The action of $G$ is called \textit{elementary} in cases (1)-(3) and \textit{non-elementary} in case
(4).
\end{itemize}

\subsection{Expander graphs and dg-bounded graphs}

Expander graph is an important object in graph theory. It will play a crucial role in our context as well. More precisely, Gromov's random monster group based on an expander sequence of graphs have Property (T). Before going into the definition of expander graph, we define Cheeger constant of a graph. 

\begin{defn}
Given a finite connected graph $\Omega$ with $|\Omega|$ vertices and a subset $A\subseteq\Omega$, denote by $\partial A$ the set of edges between $A$ and $\Omega\setminus A$. The \textit{Cheeger constant} of $\Omega$ is defined as 

\begin{align*} 
h(\Omega)=\min_{1\leq| A|\leq|\Omega|/2}  \frac{|\partial A|}{| A|}.
\end{align*}

Now, we define a expander graph. 

\begin{defn}
An \textit{expander} is a sequence of $\{\Omega_n\}_{n\in\mathbb{N}}$ of finite connected graphs with uniformly nounded degree, $|\Omega_n|\rightarrow\infty$ as $n\rightarrow\infty$, and $h(\Omega_n)\geq c$ uniformly over $n\in\mathbb{N}$ for some constant $c>0$.
\end{defn}

The concept of (combinatorial) expander can be generalized as $p$-expander with respect to a metric space as follows: 

\begin{defn}
We say that a sequence of $d$-regular finite connected graphs $\{\Omega_n=(V_n,E_n)\}_{n=1}^\infty$ is a $p$-expander with respect to the metric space $(Y,d_Y)$. if $| V_n|\rightarrow\infty$ as $n\rightarrow\infty$ and for every $f: V_n\rightarrow\infty$ we have

$$\frac{1}{| V_n|^2} \sum_{\substack{u,v\in V_n}} d_Y\big(f(u),f(v)\big)^p\leq \frac{C}{| E_n|}\sum_{\substack{\overline{uv}\in E_n}} d_Y(f(u),f(v))^p,$$
where $\overline{uv}$ denotes an edge between the vertices $u$ and $v$. 
\end{defn}

\begin{rem}
\begin{itemize}

\item[(1)] When $Y=\mathbb{R}$ and $p=2$, the above inequality is equivalent to the usual notion of (combinatorial) expander. 

\item[(2)] Every (combinatorial) expander satisfies the above inequality with $Y=l^p$ and $1< p<\infty$ (see Proposition 3 in \cite{Mat97} ).

\item[(3)] In general, the metric space $(Y, d_Y)$ is assumed to be $p$-uniformly convex which is defined as having the following property: there exists a constant $c>0$ such that for every $x,y,z\in Y$, every geodesic segment $\gamma:[0,1]\rightarrow Y$ with $\gamma(0)=y$, $\gamma(1)=z$, and every $t\in [0,1]$
we have: 
$$d_Y\big(x,\gamma(t)\big)^p\leq (1-t)d_Y(x,y)^p+t d_Y(x,z)^p-ct(1-t)d_Y(y,z)^p.$$
This include $l^p$ spaces for $1< p<\infty$.
\end{itemize}
\end{rem}

Now, we introduce the concept of  large girth, logarithmic girth and dg-bounded graphs which are important in Theorem \ref{nonelmhyp}, Theorem \ref{propt} and Theorem \ref{fixlp}. Before going into these definitions, we define girth and diameter of a graph. 

\begin{defn}
The \textit{girth} of a graph is the edge-length of its shortest non-trivial cycle. The \textit{diameter} of a graph is the greatest edge-length distance between any pair of vertices. 
\end{defn}

\begin{defn}
A sequence of graphs $\{\Omega_n\}_{n\in\mathbb{N}}$  is \textit{large girth} if $girth(\Omega_n)\rightarrow\infty$ as $n\rightarrow\infty$ and is \textit{logarithmic girth} if there exists a constant $C >0$ such that for all $n\in\mathbb{N}$
$$girth(\Omega_n)\geq C\hspace*{1mm} log(|\Omega_n|).$$
\end{defn}

\begin{defn}
A sequence of graphs $\{\Omega_n\}_{n\in\mathbb{N}}$  is \textit{dg-bounded} if there exists a constant $D>0$ such that for all $n\in\mathbb{N}$
$$\frac{diam(\Omega_n)}{girth(\Omega_n)}\leq D. $$
\end{defn}
\end{defn}

\begin{rem}
\begin{itemize}
\item[(1)] It is easy to see that logarithmic girth graphs are of large girth.
\item[(2)] If $\{\Omega_n\}_{n=1}^\infty$ has uniformly bounded degree and dg-bounded, then it is of large girth. 
For the proof of this fact, we refer the readers to \cite{AT19} (page no. 5).
\end{itemize}
\end{rem}

\begin{rem}\label{idealseq}

In our main theorems, the sequence of finite connected graphs $\{\Omega_n\}_{n=1}^\infty$ should have the following properties: 

\begin{itemize}
\item[(a)] the maximum degree of $\Omega_n$ is less than equal to $d$;
\item[(b)]$\{\Omega_n\}_{n=1}^\infty$ is dg-bounded; 
\item[(c)] $\{\Omega_n\}_{n=1}^\infty$ is of large girth or 
\item[(c')] $\{\Omega_n\}_{n=1}^\infty$ is of logarithmic girth (which implies (c));
\item[(d)] $\{\Omega_n\}_{n=1}^\infty$ is a sequence of expander graphs (which implies that $\{\Omega_n\}_{n=1}^\infty$ is a $p$-expander with respect to the $p$-uniformly convex space $l^p$ for all $p\in (1,\infty)$). 
\end{itemize}
\end{rem}

\begin{ex} 
We describe two constructions of a sequence of graphs with the properties described in Remark \ref{idealseq}.

\begin{itemize}

\item[(1)]
 The first construction is due to Margulis (see \cite{Mar82}). We consider the Cayley graphs $ \{\Omega_p=Cay\big(SL_2(\mathbb{Z}/p\mathbb{Z}),\{A_p,B_p\}\big)\}$, where  
 $p$ runs over all odd primes and $A_p$ and $B_p$ are the following two matrices, respectively:
 \[
\begin{bmatrix}
    \bar{1}  &  \bar{2}      \\
    \bar{0}  &  \bar{1}      
\end{bmatrix}
, 
\begin{bmatrix}
    \bar{1}  & \bar{0}      \\
    \bar{2}  &  \bar{1}      
\end{bmatrix} 
\]
It is a sequence of expander graphs. It satisfies $girth (\Omega_p)\geq C log|\Omega_p|$ for all $p$ and for $C > 0$ (independent of $p$). Therefore $girth(\Gamma_p)\rightarrow\infty$ as $p\rightarrow\infty$ and it is dg-bounded  by Selberg's theorem (see \cite{Sel65}) .

\item[(2)] 
The second construction is the famous Ramanujan graphs, constructed by Lubotzky-Phillips-Sarnak in \cite{LPS88}. The Cayley graph $X^{p,q}$ of the projective general linear group $PGL_2(\mathbb{F}_q)$ over the field of $q$ elements for a particular set of $(p + 1)$ generators, where $p$ and $q$ are distinct primes congruent to $1$ modulo $4$ with the Legendre symbol $(\frac{p}{q})=-1$, satisfies the following properties:
\begin{itemize}
\item[(1)] $X^{p,q}$ is $(p+1)$-regular on $N=q(q^2-1)$ vertices;
\item[(2)] $girth(X^{p,q})\geq 4\hspace*{1mm} log_pq-log_p4$;
\item[(3)] $\{X^{p,q}\}_{q=1}^\infty$ is a family of Ramanujan graphs, in particular, it is a sequence of expander graphs;
\end{itemize}

For a fixed prime $p$ and each $q$ as above, we set $\Omega_q=X^{p,q}$. Then $\{\Omega_q\}_q$ is an expander and $girth(\Omega_q)\rightarrow\infty$ as $q\rightarrow\infty$. The expander mixing lemma ensures that $\Omega_q$ is of diameter $O(logN)$. Thus, $\{\Omega\}_q$ is a dg-bounded expander. 

\end{itemize}

\end{ex}

\subsection{Gromov's random monster group}

Denote by $F_k$ the free group on the symmetric set of generators $T$ of size $2k$, and let $\Omega$ be a graph. A symmetric $F_k$-labelling $\alpha$ of $\Omega$ is a map
from the edge set of $\Omega$ to $F_k$ so that $\alpha(e^{-1})={\alpha(e)}^{-1}$, for every edge $e$. For $j\in \mathbb{N}$, we call $\mathcal{A}(\Omega, T^j)$ the set of symmetric $F_k$-labelings with values in $T^j$. Consider $\Omega$ a disjoint union of finite connected graphs $\Omega_n$, i.e. $\Omega= \sqcup_{n\in\mathbb{N}} \Omega_n$, and endow $\mathcal{A}(\Omega, T^j)$ with the product distribution coming from the uniform distributions on $\mathcal{A}(\Omega_n, T^j)$. For $\alpha\in \mathcal{A}(\Omega, T^j)$, define $\Gamma_\alpha$ to be the quotient of $F_k$ by all the words labelling closed paths in $\Omega$. Throughout the article, by $\{\Omega_n\}_{n=1}^\infty$ we will assume the sequence mentioned in Remark \ref{idealseq}, if it is not explicitly mentioned. We denote Gromov's random monster group by the symbols $\{\Gamma_\alpha\}_{\alpha\in \mathcal{A}(\Omega, T^j)}$ or $\{\Gamma_\alpha\}$ or $\Gamma(\Omega, T^j)$ depending on the contexts.

Now, we state one of the most important ingredients of Theorem \ref{mcg}, Theorem \ref{outfn}, Theorem \ref{acylin},  and Theorem \ref{hhg} . 

\begin{thm}\label{nonelmhyp} (Theorem 1 in \cite{GST20})
 Let $\Omega_n$ be a sequence of connected finite graphs, of vertex-degree between 3 and $d$ for some fixed $d \geq 3$. Assume, $|\Omega_n|\rightarrow\infty$ and that there exists $C>0$ so that $diam(\Omega_n)\leq  C girth(\Omega_n)$
for all $n\in\mathbb{N}$.
Then, for every $j\geq 1$ and almost every $\alpha\in \mathcal{A}(\Omega, T^j)$, we have that $\Gamma_\alpha$ cannot act non-elementarily on any geodesic Gromov hyperbolic space.
\end{thm}

It is due to Gromov that the random monster group has Kazhdan's Property (T), which is an important fact for proving Theorem \ref{mcg}, Theorem \ref{outfn}, Theorem \ref{acylin},  Theorem \ref{hhg} and Theorem \ref{K-amenable}. 

\begin{thm}\label{propt}( Theorem 1.2.A in \cite{Gro03},Theorem 2.17 and Corollary 2.19 in \cite{Sil03})
 Let $\{\Omega_n\}_{n=1}^\infty$ be a sequence of connected finite graphs, of vertex-degree between 3 and $d$ for some fixed $d \geq 3$ satisfying  $girth(\Omega_n)\rightarrow\infty$ as $n\rightarrow\infty$ and $\{\Omega_n\}_{n=1}^\infty$ is a sequence of expander graphs. Then $\Gamma_\alpha$ has Property (T) for almost every $\alpha\in \mathcal{A}(\Omega, T^j)$. 
\end{thm}

Using Property (T) of $\Gamma_\alpha$, we can prove the following corollary. We will also this corollary in the proofs Theorem \ref{mcg}, Theorem \ref{outfn}, Theorem \ref{acylin},  Theorem \ref{hhg}

\begin{cor}\label{surjectZ}
Let $H_\alpha$ be a finite index subgroup of $\Gamma_\alpha$. Then $H_\alpha$ does not surject onto $\mathbb{Z}$ for a.e. $\alpha$. 
\end{cor}

Now, we state a generalization of Theorem \ref{propt} which will be useful for proving Theorem \ref{a-lp-menability}.

\begin{thm}\label{fixlp}(Theorem 1.1 in \cite{NS11})
Assume that a geodesic metric space $(Y, d_Y)$ is $p$-uniformly convex and admits a sequence of logarithmic girth $p$-expanders $\{\Omega_n = (V_n, E_n)\}_{n=1}^
\infty$. Then for all $d\geq 3$, $k\geq 2$ and $j\geq 1$,  any isometric action of $\Gamma_\alpha$ on $Y$ has a common fixed point for almost every $\alpha\in \mathcal{A}(\Omega, T^j)$.
\end{thm}

\subsection{Mapping Class Group} 

Let $S_{g,b}$ be a compact, oriented, connected surface of genus $g$ with $b$ boundary components, where $g,b\in\mathbb{Z}_{\geq 0}$. We define $Homeo^+(S_{g,b})$ as the set of orientation-preserving homeomorphisms from $S_{g,b}$ to itself fixing the boundary components set-wise. The mapping class group of a closed surface is defined as follows. 

\begin{defn}
\textit{Mapping Class Group} of $S_{g,b}$ is defined by the following group $Homeo^+(S_{g,b})/\sim$, where $f\sim g$ if $f$ and $g$ are isotopic, i.e., there is a homotopy $F: S_{g,b}\times [0,1]\rightarrow S_{g,b}$ so that $F(\cdot,0)=f$, $F(\cdot, 1)=g$ and $F(\cdot,t)$ is a homeomorphism fixing the boundary components set-wise for all $t$.
\end{defn}

Now, we introduce a very important concept in the theory of mapping class group. 

\begin{defn}
The \textit{curve complex} $C(S_{g,b})$ is an abstract simplicial complex associated to a surface $S_{g,b}$. Its 1-skeleton is given by the
following data: Vertices - There is one vertex of $C(S_{g,b})$ for each isotopy class of
essential simple closed curves in $S_{g,b}$.
Edges- There is an edge between any two vertices of $C(S_{g,b})$ corresponding to isotopy classes $a$ and $b$ with geometric intersection number of $a$ and $b$ being zero.
More generally, $C(S_{g,b})$ has a $k$-simplex for each $(k + 1)$-tuple of vertices
where each pair of corresponding isotopy classes has geometric intersection
number zero.
\end{defn}

The following theorem is an excellent result by Masur-Minsky:

\begin{thm}\label{mcgcchyp}(Theorem 1.1 in \cite{MM99})
The curve complex $C(S_{g,b})$ is a $\delta$-hyperbolic metric space, where $\delta$ depends on $S_{g,b}$. 
\end{thm}

The following definition will be required in stating Theorem \ref{mcg3act}. 

\begin{defn}
We shall call a subgroup $G$ of $MCG(S_{g,b})$ \textit{reducible} if there is a one dimensional submanifold
$C$ of $S_{g,b}$, consisting of a finite, non-empty, system of disjoint, non-peripheral, simple closed curves on $S_{g,b}$  such that for any $f\in G$ there is a homeomorphism $F: S_{g,b}\rightarrow S_{g,b}$ in the isotopy class of $f$ that leaves $C$ invariant (that is, such that $F(C)=C$). 
\end{defn}

Now, we state some theorems which will be useful for proving Theorem \ref{mcg}.

\begin{thm}
An element $\phi$ of the mapping class group $MCG(S_{g,b})$ acts loxodromically on $C(S_{g,b})$
if and only if $\phi$ is `pseudo-Anosov'. 
\end{thm}

\begin{thm}\label{mcg3act}(\cite{Iva92}, Theorem 4.6 in \cite{MP89})
Every subgroup $H \leq MCG(S_{g,b})$ either
\begin{itemize}
\item contains two `pseudo-Anosov diffeomorphisms' of $S_{g,b}$ that generate a rank two free
subgroup of $H$, or
\item  is virtually cyclic and virtually generated by a `pseudo-Anosov diffeomorphism', or
\item $H$ is reducible.
\end{itemize}
\end{thm}

We end this subsection by introducing a concept which will be useful in the proof of Theorem \ref{mcg}.

\begin{defn}
We define \textit{complexity} of a surface, denoted by $\tau(S_{g,b})$, by the quantity $(3g-3+b)$. 
\end{defn}

\begin{rem}\label{complexity0}
It is not difficult to see that $\tau(S)\leq 0$ if and only if  $S$ is one of the following surfaces: annulus, sphere, pair of pants, disc or torus. 
\end{rem}

\subsection{Braid group}

 The \textit{braid group} on $n$ strands (denoted $B_n$),  is the group whose elements are equivalence classes of $n$-braids (e.g. under ambient isotopy), and whose group operation is composition of braids. 
 
 \begin{prop}\label{braidmcg}
 Let $S_{n+1}$ be the sphere with $(n+1)$ punctures and $MCG(S_{n+1})$  be the mapping class group of $S_{n+1}$. We also assume that $MCG_{x}(S_{n+1})$ denotes the subgroup of $MCG(S_{n+1})$ which fixes a fixed puncture $x$. Then we have the following short exact sequence:
 $$1\rightarrow\mathbb{Z}\rightarrow B_n\rightarrow MCG_x(S_{n+1})\rightarrow 1 $$
 \end{prop}

\subsection{Outer automorphism group of a free group} 

We start this subsection with the definition of outer automorphism group of a free group.

\begin{defn}
The outer automorphism group of a finitely generated free group $F_N$, is the quotient, $Aut(F_N) / Inn(F_N)$, where $Aut(F_N)$ is the automorphism group of $F_N$ and $Inn(F_N)$ is the subgroup consisting of inner automorphisms. This group will be denoted by $Out(F_N)$.
\end{defn}

Now, we introduce the concept of free factor complex $FF_N$. It was originally introduced by Hatcher and Vogtmann in \cite{HV98}. 

\begin{defn}
The \textit{free factor complex} $FF_N$ is an abstract simplicial complex associated to a free group $F_N$.
The set of vertices $V(FF_N)$ of $FF_N$ is defined as the set of all $F_N$-conjugacy classes $[A]$ of proper free factors $A$ of $F_N$. Two distinct vertices $[A]$ and $[B]$ of $FF_N$ are joined by an edge whenever there exist proper free factors $A, B$ of $F_N$ representing $[A]$ and $[B]$ respectively, such that either
$A \leq B$ or $B \leq A$. More generally, for $k \geq 1$, a collection of $(k + 1)$ distinct vertices $[A_0], \cdots , [A_k]$ of $FF_N$ spans a $k$-simplex in $FF_N$ if, up to a possible re-ordering of these vertices there exist representatives $A_i$ of $[A_i]$ such that $A_0 \leq A1 \leq \cdots \leq A_k$.
There is a canonical action of $Out(F_N)$ on $FF_N$ by simplicial automorphisms: If $\Delta=
{[A_0], . . . , [A_k]}$ is a $k$-simplex and $\phi\in Out(F_N )$, then $\phi(\Delta) := {[\phi(A_0)], \cdots , [\phi(A_k)]}$. 
\end{defn}

The following result has been obtained by Bestvina-Feign. 

\begin{thm}\label{FFNhyp}\cite{BF14} 
The free factor complex $FF_N$ is hyperbolic.
\end{thm}

The following concept will be required in our next results. 

\begin{defn} 
An automorphism $\phi\in Out(F_N)$ is fully \textit{irreducible} if no non-trivial power of $\phi$ preserves the conjugacy class of a proper free factor of $F_N$. 
\end{defn}

Now, we state some theorems which will be useful for proving Theorem \ref{outfn}.

\begin{thm} 
An automorphism $\phi\in Out(F_N)$ acts loxodromically on $FF_N$ if and only if $\phi$ is fully irreducible. 
\end{thm}

\begin{thm}\label{outfn3act}\cite{HM09}
Every subgroup of $Out(F_N)$ (finitely generated or not) either
\begin{itemize}
\item contains two fully irreducible automorphisms that generate a rank two free subgroup, or
\item is virtually cyclic and virtually generated by a fully irreducible automorphism, or
\item virtually fixes the conjugacy class of a proper free factor of $F_N$ .
\end{itemize}
\end{thm}

\begin{thm}\label{subIAn}(Corollary 2.9 in \cite{BW11})
Let $\overline{IA_N}$ be the kernel of the map $Out(F_N)\rightarrow GL_N(\mathbb{Z})$ given by the action of $Out(F_N)$ on the first homology of $F_N$. Then every non-trivial subgroup of $\overline{IA_N}$
maps onto $\mathbb{Z}$. 
\end{thm}

\subsection{Acylindrically Hyperbolic Group} 
In this subsection, we first define acylindrical action of a group. 

\begin{defn}
The action of a group $G$ on a metric space $X$ is called \textit{acylindrical} if for every
$\epsilon >0$ there exist $R,N >0$ such that for every two points $x, y$ with $d(x, y) \geq R$,
there are at most $N$ elements $g\in G$ satisfying $d(x, gx) \leq \epsilon$ and $d(y,gy) \leq \epsilon$.
\end{defn}

Now, we define acylindrically hyperbolic group. 

\begin{defn}
 We call a group $G$ \textit{acylindrically hyperbolic} if $G$ admits a non-elementary acylindrical action on a hyperbolic space. 
\end{defn}

\begin{ex}
The examples of acylindrically hyperbolic groups are as follows (see \cite{Osi16})
\begin{itemize}
\item non-elementary hyperbolic groups and relatively hyperbolic groups;
\item $MCG(S_{g,b})$ unless $g=0$ and $p\leq 3$;
\item $Out(F_N)$ ($N\geq 2$);
\item directly indecomposable right angled Artin groups;
\item most 3-manifold groups.
\end{itemize}
\end{ex}

Now, we state the following theorem which will be useful in the proof of Theorem \ref{acylin}:

\begin{thm}\label{acy3act}\cite{Osi16}
 Let $G$ be a group acting acylindrically on a hyperbolic space. Then $G$ satisfies
exactly one of the following three conditions.
\begin{itemize}
\item $G$ has bounded orbits,
\item $G$ is virtually cyclic and contains a loxodromic element,
\item $G$ contains infinitely many independent loxodromic elements.
\end{itemize}
\end{thm}

\subsection{Hierarchically Hyperbolic Group} Before going into the definition of Hierarchically Hyperbolic Group, we define Hierarchically Hyperbolic Space. The definition of Hierarchically Hyperbolic space is taken from \cite{BHS17}. 

\begin{defn}
The metric space $(\mathcal{X}, d_{\mathcal{X}})$ is a \textit{Hierarchically Hyperbolic Space}(HHS) if there exists $\delta\geq 0$, an index set $\mathcal{S}$, and a set $\{\hat{\mathcal{C}}W: W\in \mathcal{S}\}$ of
$\delta$-hyperbolic spaces, such that the following conditions are satisfied:
\begin{itemize}
\item[(1)] (\textbf{Projections}) There is a set $\{\pi_W:\mathcal{X}\rightarrow 2^{\hat{\mathcal{C}}W}:W\in \mathcal{S}\}$  of projections sending
points in $\mathcal{X}$ to sets of diameter bounded by some $\xi\geq 0$ in the various $\hat{\mathcal{C}}W$, $W\in \mathcal{S}$.
\item[(2)] (\textbf{Nesting}) $\mathcal{S}$ is equipped with a partial order $\sqsubseteq$, and either $S=\phi$ or $S$ contains a unique $\sqsubseteq$-maximal element; when $V \sqsubseteq W$, we say $V$ is nested in $W$. We require that $W\sqsubseteq W$ for all $W\in\mathcal{S}$. For each $W\in \mathcal{S}$, we denote by $\mathcal{S}_W$ the set of $V\in\mathcal{S}$ such that $V\sqsubseteq W$. Moreover, for all $V,W\in\mathcal{S}$ with $V$ properly nested into $W$, there is a specified subset
$\rho_W^V\subset \hat{\mathcal{C}}W$ with $diam_{\hat{\mathcal{C}}W}(\rho_W^V)\leq\xi$. There is also a \textit{projection} $\rho_V^W: \hat{\mathcal{C}}W\rightarrow 2^{\hat{\mathcal{C}}V}$. (The similarity in notation is justified by viewing $\rho_W^V$ as a coarsely constant map $\hat{\mathcal{C}}V\rightarrow 2^{\hat{\mathcal{C}}W}$.)

\item[(3)] (\textbf{Orthogonality}) $\mathcal{S}$ has a symmetric and anti-reflexive relation called \textit{orthogonality}: we write $V \perp W$ when $V, W$ are orthogonal. Also, whenever $V \sqsubseteq W$ and $W\perp U$, we
require that $V\perp U$. Finally, we require that for each $T\in\mathcal{S}$ and each $U\in\mathcal{S}_T$ for which $\{V\in \mathcal{S}_T : V\perp U\}\neq \phi $, there exists $W\in \mathcal{S}_T-\{T\}$, so that whenever $V\perp U$ and $V\sqsubseteq T$, we have $V\sqsubseteq W$. Finally, if $V\perp W$, then $V, W$ are not $\sqsubseteq$-comparable. 

\item[(4)] (\textbf{Transversality and consistency}) If $V, W\in\mathcal{S}$ are not orthogonal and neither is nested in the other, then we say $V,W$ are transverse, denoted $V\pitchfork W$. There exists $k_0\geq 0$ such that if $V\pitchfork W$, then there are sets $\rho_W^V\subseteq \hat{\mathcal{C}}W$ and $\rho_V^W\subseteq \hat{\mathcal{C}}V$ each of diameter at most $\xi$ and satisfying: 

$$min\{d_{\hat{\mathcal{C}}W}\big(\pi_W(x),\rho_W^V\big), d_{\hat{\mathcal{C}}V}\big(\pi_V(x),\rho_V^W\big) \}\leq \kappa_0 $$

for all $x\in\mathcal{X}$; alternatively, in the case $V\sqsubseteq W$, then for all $x\in\mathcal{X}$ we have:

$$min\{d_{\hat{\mathcal{C}}W}\big(\pi_W(x),\rho_W^V\big),  diam_{\hat{\mathcal{C}}V}\big(\pi_V(x)\cup \rho_V^W(\pi_W(x))\big)\}\}\leq \kappa_0.$$

Suppose that: either $U \not\sqsubseteq V$ or $U\pitchfork V$, and either $U\not\sqsubseteq W$ or $U\pitchfork W$. Then we have: if 
$V\pitchfork W$, then 

$$min\{d_{\hat{\mathcal{C}}W}\big(\rho^U_W,\rho_V^W\big), d_{\hat{\mathcal{C}}V}\big(\rho^U_V,\rho_W^V\big) \}\leq \kappa_0 $$

and if $V\not\sqsubseteq W$, then 

$$min\{d_{\hat{\mathcal{C}}W}\big(\rho^U_W,\rho^V_W\big),  diam_{\hat{\mathcal{C}}V}\big(\rho^U_V\cup \rho^W_V(\rho^U_W)\big)\}\leq \kappa_0.$$

Finally, if $V \sqsubseteq U$ or $U\perp V$, then$d_{\hat{\mathcal{C}}W}(\rho^U_W,\rho^V_W)\leq \kappa_0$
 whenever $W\in \mathcal{S}-\{U,V\}$ satisfies either $V\sqsubseteq W$ or $V\pitchfork W$ and either $U\sqsubseteq  W$ or $U\pitchfork W$.

\item[(5)] (\textbf{Finite complexity}) There exists $n\geq 0$, the complexity of $\mathcal{X}$ (with respect to $\mathcal{S}$), so that any sequence $(U_i)$ with $U_i$ properly nested into $U_{i+1}$ has length at most $n$. 

\item[(6)] (\textbf{Distance Formula}) There exists $s_0\geq \xi$ such that for all $s\geq s_0$ there exist constants $K,C$ such that for all $x,x'\in\mathcal{X}$,

$$d_{\mathcal{X}}(x,x')\asymp_{(K,C)}\sum_{W\in\mathcal{S}}   \{\{d_{\hat{\mathcal{C}}W}\big( \pi_W(x), \pi_W(x') \big)   \}\}_s.$$

We often write $\sigma_{\mathcal{X},s}(x,x')$ to denote the right-hand side of Item (6); more generally, given $W\in\mathcal{S}$, we denote $\sigma_{W,s}(x,x')$ the corresponding sum taken over $\mathcal{S}_W$.

\item[(7)] (\textbf{Large Links}) There exists $\lambda\geq 1$ such that the following holds. Let $W\in\mathcal{S}$ and let $x,x'\in\mathcal{X}$. Let $N=\lambda d_{\hat{\mathcal{C}}W}\big(\pi_W(x),\pi_W(x')\big)+\lambda$. Then there exists $\{T_i\}_{i=1,\cdots,N}\subseteq \mathcal{S}_W-\{W\}$ such that for all $T\in \mathcal{S}_W-\{W\}$, either $T\in\mathcal{S}_{T_i}$ for some $i$, or $d_{\hat{\mathcal{C}}}T\big(\pi_T(x),\pi_T(x')\big)<s_0$. Also, $d_{\hat{\mathcal{C}}}W\big(\pi_W(x),\rho_W^{T_i})\leq N$ for each $i$. 

\item[(8)] (\textbf{Bounded Geodesic Image}) For all $W\in\mathcal{S}$, all $V\in \mathcal{S}_W-\{W\}$, and all geodesics $\gamma$ of $\hat{\mathcal{C}}W$,
either $diam_{\hat{\mathcal{C}}V}\big(\rho_V^W(\gamma)\big)\leq B$ or $\gamma\cap\mathcal{N}_E(\rho_W^V)\neq \emptyset$ for some uniform $B,E$.

\item[(9)] (\textbf{Realization}) For each $\kappa$ there exists $\theta_e, \theta_u$ such that the following holds. Let $\overrightarrow{b}\in\prod_{W\in\mathcal{S}} 2^{\hat{\mathcal{C}}W}$ have each coordinate correspond to a subset of $\hat{\mathcal{C}}W$ of diameter at most $\kappa$; for each $W$, let $b_W$ denote the $\hat{\mathcal{C}}W$-coordinate of $\overrightarrow{b}$. Suppose that whenever $V\pitchfork W$ we have 

$$min\{d_{\hat{\mathcal{C}}W}(b_W,\rho_W^V),  d_{\hat{\mathcal{C}}V}(b_V,\rho_V^W)\}\leq \kappa $$

and whenever $V\sqsubseteq W$ we have 

$$ min\{d_{\hat{\mathcal{C}}W}(b_W,\rho_W^V),  diam_{\hat{\mathcal{C}}V}(b_V\cup\rho_V^W(b_W))\}\leq \kappa$$

Then there the set of all $x\in\mathcal{X}$ so that $d_{\hat{\mathcal{C}}W}(b_W,\pi_W(x))\leq \theta_e$ for all $\hat{\mathcal{C}}W\in\mathcal{S}$
is non-empty and has diameter at most $\theta_u$. ( A tuple $\overrightarrow{b}$ satisfying the inequalities above is called \textit{$\kappa$-consistent}.)

\item[(10)] (\textbf{Hierarchy paths}) There exists $D\geq 0$ so that any pair of points in $\mathcal{X}$ can be joined by a $(D,D)$-quasi-geodesic $\gamma$ with the property that, for each $W\in\mathcal{S}$, the projection $\pi_W(\gamma)$ is at Hausdorff distance at most $D$ from any geodesic connecting $\pi_W(x)$ to $\pi_W(y)$. We call such quasi-geodesics \textit{hierarchy paths}. 

\end{itemize}

\end{defn}

Now, we define hierarchically hyperbolic group.

\begin{defn}
A finitely generated group $G$ is a \textit{Hierarchically Hyperbolic Group} (HHG) if there exists a hierarchically hyperbolic space $(\mathcal{X},\mathcal{S})$ such that $G\leq Aut(\mathcal{S})$, the quasi-action of $G$ on $\mathcal{X}$ is proper and cobounded, and $G$ acts on $\mathcal{S}$ with finitely many orbits. We refer the reader to Subsection 1.2 of \cite{DHS17} for the detailed definition of 
$Aut(\mathcal{S})$ the action of $Aut(\mathcal{S})$ on $\mathcal{X}$.
\end{defn}

\begin{ex} 
The examples of hierarchically hyperbolic groups are given below:
\begin{itemize}
\item[(a)] mapping class group of a connected, oriented surface of finite type;
\item[(b)] hyperbolic groups;
\item[(c)]some relatively hyperbolic groups;
\item[(d)]many groups acting geometrically on $CAT(0)$ cube complexes, in particular, right-angled Artin groups.
\end{itemize}

\end{ex}

Now, we state one theorem, two propositions and one lemma which will be useful in the proof of Theorem \ref{outfn}.

\begin{thm}\label{hhgacyact}( Theorem K in \cite{BHS17})
Let $(\mathcal{X}, \mathcal{S})$ be a hierarchically hyperbolic space and  $G\leq Aut(\mathcal{S})$ act properly and cocompactly on $\mathcal{X}$ . Let $S$ be the maximal element of $\mathcal{S}$ and denote by $\hat{\mathcal{C}}S$ the corresponding hyperbolic space. Then $G$ acts acylindrically on $\hat{\mathcal{C}}S$.
\end{thm}

\begin{prop}\label{hhg2act}(Proposition 9.2 in \cite{DHS17})
Let $(\mathcal{X},\mathcal{S})$ be an HHS with $\mathcal{X}$ proper and $\mathcal{S}$ countable. Let the countable
group $G \leq Aut(\mathcal{S})$ act with unbounded orbits in $\mathcal{X}$ and without a global fixed point in the boundary $\partial \hat{\mathcal{C}}S$ of $\hat{\mathcal{C}}S$.
Then either $G$ contains an irreducible axial element, or there exists $U\in \mathcal{S}-\{S\}$  so that $| G\cdot U| <\infty$.
\end{prop}

\begin{prop}\label{hhgtitsalt}(Proposition 9.16 in \cite{DHS17})
 Let $G$ be be a hierarchically hyperbolic group. Then any $H \leq G$
containing an irreducible axial element is virtually $\mathbb{Z}$ or contains a non-abelian free group.
\end{prop}

\begin{lem}\label{hhgsubfin}(Lemma 9.17 in \cite{DHS17})
 Let $G$ be an HHG and $(\mathcal{X},\mathcal{S})$ be an HHS associated with $G$. We assume that $S \in\mathcal{S}$ is the maximal element. Suppose that $H\leq G$ has bounded orbits in $\hat{\mathcal{C}}S$ and fixes some point $p$ in the boundary $\partial\hat{\mathcal{C}}S$ of $\hat{\mathcal{C}}S$. Then $| H|<\infty$.
\end{lem}

\subsection{Groups with Property $F_{L^p}$ and a-$FL^p$-menability }

\begin{defn}
A discrete group $G$ has property $F_{L^p}$ if any  affine isometric action on $L^p$ has a fixed point. 
\end{defn}

By Theorem \ref{fixlp}, we obtain that Gromov's random monster has Property $F_{L^p}$ for almost every $\alpha\in\mathcal{A}(\Omega,T^j)$ and for $1< p<\infty$. Now, we introduce Property a-$F_{L^p}$-menability which is, in some sense, opposite to Property $F_{L^p}$. 

\begin{defn}\label{defnalpmenable}
A discrete group $G$ has \textit{Property a-$F_{L^p}$-menability} if there 
a proper affine isometric action of $G$ on an $L^p$ space for some $p\in (1,\infty)$. 
\end{defn}

The examples of groups with a-$F_{L^p}$-menability are as follows: 
\begin{itemize}
\item Groups with Haagerup Property (i.e. groups having proper affine isometric action on a Hilbert space), for example free group, lattices in $SL_2(\mathbb{R}$, $SO(n,1)$, $SU(n,1)$ etc., see \cite{CCJ$^{+}$01} ;
\item Lattices in $Sp(n,1)$ (this is implicit in \cite{Pan90}, but the details can be found in \cite{MS20});
\item Hyperbolic groups (due to G. Yu, see \cite{Yu05} )
\item Some relatively hyperbolic groups (see \cite{CD20}, \cite{GRT}).
\end{itemize}

\subsection{K-amenable groups}

Another weak form of amenability is K-theoretic amenability (or K-amenability for short). We give a rough version of the definition of K-amenable group. For the detailed definition of K-amenability we refer the readers to Definition 2.2 of \cite{Cun83}. Now, let us consider a discrete countable group $G$ and the epimorphism $\lambda_G: C^*G\rightarrow C^*_r G$ induced by the left-regular representation of $G$, where $C^*G$ is the maximal $C^*$-algebra of $G$ and $C^*_r G$ is the reduced $C^*$-algebra of $G$ . A characterization of amenability is that $G$ is amenable if and only if $\lambda_G$ is an isomorphism.  Roughly speaking we say that $G$ is K-amenable if  $\lambda_G$ induces isomorphisms in K-theory, i.e.,
$$(\lambda_G)_*: K_i(C^*G)\rightarrow K_i(C^*_r)$$
is an isomorphism for $i=0,1$. The real definition in \cite{Cun83} implies the above condition. 

\begin{ex}
The examples of K-amenable groups are as follows:
\begin{itemize}
\item[(1)] groups with Haagerup property ( see \cite{Tu99}); 
\item[(2)] 1-relator groups (see \cite{BBV99}); 
\item[(3)] the fundamental groups of Haken 3- manifolds, this class contains all  knot groups (see \cite{BBV99});
\item[(4)] $\mathbb{Z}^2 \rtimes SL_2(\mathbb{Z})$ (see \cite{JV84}).
\end{itemize}
\end{ex}

We state a theorem and a proposition which will be useful to prove Theorem \ref{K-amenable}.

\begin{thm}\label{kamensub}(Theorem 2.4. (a) in \cite{Cun83})
Let $G$ be a countable discrete group. If $G$ is K-amenable, then all subgroups of $G$ are K-amenable. 
\end{thm}

\begin{thm}\label{kamen+t} ( Remark 2.7.(b) of \cite{Cun83}, \cite{JV84}, Corollary 3.7 in \cite{JV84})
Any K-amenable discrete group with Property (T) is finite.
\end{thm}

\subsection{Hereditary super-rigidity and the stable properties of super-rigidity}

It is easy to see that if  $\Gamma_\alpha$ is super-rigid with respect to $G$ and $H\leq G$, then  $\Gamma_\alpha$ is also super-rigid with respect to $H$. This property possibly does not passes to the quotient group. But, we do not have any counter example. Now, we introduce hereditary super-rigidity of $\Gamma_\alpha$ which is a stronger property that super-rigidity. 

\begin{defn}
We say that $\Gamma_\alpha$  has \textit{hereditary super-rigidity} with respect to a group $G$ if any finite index subgroup $\Gamma'_\alpha$ has super-rigidity with respect to $G$ for a.e. $\alpha\in\mathcal{A}(\Omega,T^j)$.
\end{defn}

\begin{ex}\label{hsexo}
$\Gamma_\alpha$ has hereditary super-rigidity with respect to a-$FL^p$-menable group. We will prove this fact in Section 3.  
\end{ex}

Now, we study the behaviour of super-rigidity under a short exact sequence. 

\begin{thm}\label{shortexact}
Let $1\rightarrow N\xrightarrow{i} G\xrightarrow{q} G/N\rightarrow 1$ be a short exact sequence. If $\Gamma_\alpha$ has hereditary super-rigidity with respect to $N$ and super-rigidity with respect to $G/N$, then it has super-rigidity with respect to $G$. 
\end{thm}

\begin{proof} 
Let $\phi_\alpha:\Gamma_\alpha\rightarrow G$ be a group homomorphism. We consider the homomorphism $q\circ \phi_\alpha: \Gamma_\alpha\rightarrow G/N$. Since $\Gamma_\alpha$ has super-rigidity with respect to $G/N$,  $(q\circ \phi_\alpha)( \Gamma_\alpha)$ is finite. Now, we consider the subgroup $\phi_\alpha(\Gamma_\alpha)\cap i(N)$. Let $\Gamma'_\alpha$ be the inverse image of this subgroup under the map $\phi_\alpha$. Using the fact that $(q\circ \phi_\alpha)( \Gamma_\alpha)$ is finite, it is easy to see that $\Gamma'_\alpha$ is finite index subgroup of $\Gamma_\alpha$. Since $\Gamma_\alpha$ has hereditary super-rigidity with respect to $N$, $\phi_\alpha(\Gamma'_\alpha)$ is finite. Therefore, $\phi_\alpha(\Gamma_\alpha)$ is finite. 
\end{proof}

\section{Proofs of main theorems and corollaries}

\textbf{Proof of Theorem \ref{mcg}:} We will prove by induction on the complexity $\tau(S)$ of a surface $S$. From Remark \ref{complexity0}, we obtain that $\tau(S)\leq 0$ if and only if  $S$ is one of the following surfaces: annulus, sphere, pair of pants, disc or torus. If $S$ is a annulus, sphere, pair of pants or disc, $MCG(S)$ is trivial;  if $S$ is a torus, $MCG(S)$ is $SL_2(\mathbb{Z})$. In first case, the theorem is trivially true. Since $SL_2(\mathbb{Z})$ are a-$L^p$-menable, we obtain from Theorem \ref{a-lp-menability} that the theorem is true.

Now, we assume that $S_{g,b}$ is  any surface with $\tau(S_{g,b})>0$ and the theorem is true for all
surfaces $T$ with $\tau(T)<\tau(S_{g,b})$. 
 Let $\phi_\alpha:\Gamma_\alpha\rightarrow MCG(S_{g,b})$ be a group homomorphism and $H_\alpha=\phi_\alpha(\Gamma_\alpha)$ for all $\alpha\in\mathcal{A}(\Omega,T^j)$. Let $H_\alpha$ be infinite. We will prove the theorem by contradiction. The curve-complex $\mathcal{C}(S_{g,b})$ is hyperbolic by Theorem \ref{mcgcchyp}. $H_\alpha$ acts isometrically on $\mathcal{C}(S_{g,b})$. By Theorem \ref{nonelmhyp}, $H_\alpha$ acts elementarily on $\mathcal{C}(S_{g,b})$. By Theorem \ref{mcg3act}, any subgroup of $MCG(S_{g,b})$ having an elementary action on $\mathcal{C}(S_{g,b})$ is either virtually cyclic or reducible, i.e.,  fixes the isotopy class of a simple closed curve on $S$. Since no finite index subgroup of $\Gamma_\alpha$ surjects onto $\mathbb{Z}$ for a.e. $\alpha$ (by Corollary \ref{surjectZ}), $H_\alpha$ is not virtually cyclic.
Therefore, $H_\alpha$ is reducible, which implies that $H_\alpha$ fixes a finite, non-empty, collection $C$ of disjoint, non-peripheral, simple closed curves on $S_{g,b}$. Since $H_\alpha$ leaves $C$ invariant, it induces a permutation on the connected components $\{S_1,\cdots, S_n\}$ ($n>0$) of $S_{g,b}\setminus C$. Hence, there is a homomorphism from $H_\alpha$ to $\Sigma_n$, where $\Sigma_n$ is the group of permutations on $n$ letters. Let $H_\alpha^0$ be the kernel of this homomorphism. 

Since $H_\alpha^0$ leaves each subsurface $S_i$ invariant (here each $S_i$ is a compact surface with boundary), there is a homomorphism 
$$\psi_\alpha:H_\alpha^0\rightarrow MCG(S_1)\times\cdots MCG(S_n).$$
The kernel is abelian by \cite{BLM83} (see  Lemma 2.1).  We observe that $\psi_\alpha(g)=1$ if and only if $g$ is a product of Dehn twists about curves in $C$. Let the kernel of $\psi_\alpha$ be of finite index in $H_\alpha^0$.  It implies that $H_\alpha^0$, and hence $H_\alpha$, is virtually abelian. But it can not happen since no finite index subgroup of $\Gamma_\alpha$ surjects onto $\mathbb{Z}$. Therefore, the kernel of $\psi_\alpha$ is of infinite index in $H_\alpha^0$, which implies that $\psi_\alpha(H_\alpha^0)$ is infinite. Let $p_i: MCG(S_1)\times\cdots MCG(S_n)\rightarrow MCG(S_i)$ be the natural projection for all $i=1,\cdots,n$. 
Then $(p_j\circ \psi_\alpha)(H_\alpha^0)$ is infinite 
for at least one $1\leq j\leq n$. But, by induction, since $\tau(S_j)<\tau(S_{g,b})$, $(p_j\circ \psi_\alpha)(H_\alpha^0)$ must be finite, a contradiction. \hfill\(\Box\)

\vspace{5mm}

\textbf{Proof of Corollary \ref{braid}:} From Proposition \ref{braidmcg}, we have the following short exact sequence 
$$1\rightarrow\mathbb{Z}\rightarrow B_n\rightarrow MCG_x(S_{n+1})\rightarrow 1 $$
Since $\mathbb{Z}$ is a-$FL^p$-menable, using Corollary \ref{hsa-lp-menability}  we obtain that$\Gamma_\alpha$ has hereditary super-rigidity with respect to $\mathbb{Z}$ for a.e. $\alpha$. On the other hand, since $MCG_{x}(S_{n+1})$ is a subgroup of $MCG(S_{n+1})$, using Theorem \ref{mcg} we have $\Gamma_\alpha$ has super-rigidity with respect to $MCG_{x}(S_{n+1})$ for a.e. $\alpha$. Now, using Theorem \ref{shortexact} we conclude that $\Gamma_\alpha$ is super-rigid with respect to $B_n$ for a.e. $\alpha$. Hence, we have our corollary. \hfill\(\Box\)

\vspace{5mm}

\textbf{Proof of Theorem \ref{outfn}:} Let $\phi_\alpha:\Gamma_\alpha\rightarrow Out(F_N)$ be a group homomorphism and $H_\alpha=\phi_\alpha(\Gamma_\alpha)$ for all $\alpha\in\mathcal{A}(\Omega,T^j)$. $H_\alpha$ acts isometrically on the free factor complex $FF_N$ of $Out(F_N)$.  By Theorem \ref{FFNhyp}, it is hyperbolic. Moreover, by Theorem \ref{nonelmhyp}, $H_\alpha$ acts elementarily on $FF_N$. Now, by Theorem \ref{outfn3act}, any subgroup of $Out(F_N)$ having an elementary action on $FF_N$ is either virtually cyclic or reducible, i.e., virtually fixes the conjugacy class of a proper free factor of $F_N$. Since no finite index subgroup of $\Gamma_\alpha$ surjects onto $\mathbb{Z}$ for a.e. $\alpha$ (by Corollary \ref{surjectZ}), $H_\alpha$ is not virtually cyclic. 

Let $F_N=L\ast L'$, where $L$ is a proper free factor of $F_N$ and there exist a finite index subgroup $H^0_\alpha$ of $H_\alpha$ fixing the conjugacy class of $L$. We also assume that $\psi(L)=g_\psi^{-1}Lg_\psi$ for all $\psi\in H^0_\alpha$. Observe that $\psi\mapsto\restr{\psi}{L}$ defines a homomorphism from $H^0_\alpha$ to $Out(L)$. Likewise, the action on the quotient $F_N/\langle\langle L\rangle\rangle$ induces a homomorphism $H_\alpha^0\rightarrow Out(L')$. By induction, we know that the induced action of $H_\alpha^0$ on the abelianization of both $L$ and $L'$ factors through a finite group. Thus the action of $H_\alpha^0$ on the abelianization of $F_N=L\ast L'$ lies in a block triangular subgroup (with respect to a basis that is the union of bases for $L$ and $L'$) 
\[
\begin{bmatrix} J & 0 \\
 \star & J' \end{bmatrix}
\]
inside $GL_N(\mathbb{Z})$, where $J$ and $J'$ are finite. This matrix group is finitely generated and virtually abelian, whereas $\Gamma_\alpha$, and therefore $H_\alpha^0$, does not have a subgroup of finite index that maps onto $\mathbb{Z}$. Thus the action of $H_\alpha^0$ on the homology of $F_N$ factors through a finite group, and hence that of $H_\alpha$ does too, i.e., $H_\alpha\cap\overline{IA_N}$ has finite index in $H_\alpha$, where $\overline{IA_N}$ is the kernel of the map $Out(F_N)\rightarrow GL_N(\mathbb{Z})$ given by the action of $Out(F_N)$ on the first homology of $F_N$.

Now, by Theorem \ref{subIAn}, every non-trivial  subgroup of $\overline{IA_n}$ maps onto $\mathbb{Z}$. Since no finite index subgroup of $H_\alpha$ maps onto $\mathbb{Z}$, we obtain that $H_\alpha\cap\overline{IA_N}=\{1\}$. Therefore, $H_\alpha$ is finite. \hfill\(\Box\)

\vspace{5mm}

\textbf{Proof of Corollary \ref{autfn}:} Since the center of $F_N$ is trivial, we have $F_N\cong Inn(F_N)$. Therefore, we have the following short exact sequence

$$1\rightarrow F_N\rightarrow Aut(F_N)\rightarrow Out(F_N)\rightarrow 1$$

Since $F_N$ is a-$L^p$-menable, by Theorem \ref{hsa-lp-menability} $\Gamma_\alpha$ has hereditary super-rigidity with respect to $F_N$ for a.e. $\alpha$. On the other hand, by Theorem \ref{outfn} $\Gamma_\alpha$ has super-rigidity with respect to $Out(F_N)$ for a.e. $\alpha$. Now, using Theorem \ref{shortexact} we obtain that $\Gamma_\alpha$ has super-rigidity with respect to $Aut(F_N)$ for a.e. $\alpha$. \hfill\(\Box\)

\vspace{5mm}

\textbf{Proof of Theorem \ref{acylin}:}

 Let $\phi_\alpha: \Gamma_\alpha\rightarrow G$ be a group homomorphism for a.e. $\alpha\in\mathcal{A}(\Omega,T^j)$, where $G$ is an acylindrically hyperbolic group. We denote $\phi_\alpha(\Gamma_\alpha)$ by $H_\alpha$. We consider an acylindrical action of $G$ on a hyperbolic space $X$. Then, by Theorem \ref{nonelmhyp} and Theorem \ref{acy3act}, either the action of $H_\alpha$ on $X$ is elliptic or $H_\alpha$ is virtually cyclic and contains a loxodromic element. But, by Corollary \ref{surjectZ}, $H_\alpha$ can not be virtually cyclic.  As a consequence, the action of $H_\alpha$ on $X$ is elliptic. \hfill\(\Box\)

\vspace{5mm}

\textbf{Proof of Theorem \ref{hhg}:} Let $\phi_\alpha:\Gamma_\alpha\rightarrow G$ be a group homomorphism for all $\alpha\in \mathcal{A}(\Omega, T^j)$, where $G$ is a hierarchically hyperbolic group. Suppose $(\mathcal{X},\mathcal{S})$ is a hierarchically hyperbolic space obtained from the definition of hierarchically hyperbolic group given in Subsection 2.6.  We assume that $S$ denotes the maximally nested element in $ \mathcal{S}$ and  $\hat{\mathcal{C}}S$ denotes its associated hyperbolic space. Let $H_\alpha=\phi_\alpha(\Gamma_\alpha)$. $H_\alpha$ acts isometrically on the hyperbolic space $\hat{\mathcal{C}}S$. By Theorem \ref{nonelmhyp}, $H_\alpha$ acts elementarily on $\hat{\mathcal{C}}S$. We obtain from Theorem \ref{hhgacyact} that the action of $H_\alpha$ on $\hat{\mathcal{C}}S$ is acylindrical. Therefore, by Theorem \ref{acy3act}, if $H_\alpha$ has unbounded orbits in $\hat{\mathcal{C}}S$, then $H_\alpha$ is virtually cyclic. Since no finite index subgroups $\Gamma_\alpha$ surjects onto $\mathbb{Z}$ for a.e. $\alpha$ (by Corollary \ref{surjectZ}), $H_\alpha$ is not virtually cyclic for a.e. $\alpha$. Therefore, $H_\alpha$ has bounded orbits in $\hat{\mathcal{C}}S$. 
We have two cases:
\begin{itemize}
\item[(1)] Let $H_\alpha$ fix a point $p\in\partial \hat{\mathcal{C}}S$, the boundary of $\hat{\mathcal{C}}S$. Then, by Lemma \ref{hhgsubfin}, $| H_\alpha|<\infty$. 

\item[(2)] Let $H_\alpha$ do not have any fixed boundary point. Because of Proposition \ref{hhgtitsalt}, $H_\alpha$ does not have any irreducible axial element. Now, using Proposition \ref{hhg2act}, we obtain that there exists $U\in\mathcal{S}-\{S\}$ such that $| H_\alpha\cdot U|<\infty$. Therefore, there exists $U'\in\mathcal{S}-\{S\}$ such that some finite index subgroup $H^0_\alpha$ of $H_\alpha$ fixes $U'$. Using induction on complexity, we obtain that $H_\alpha$ is  finite for a.e. $\alpha$. 
\end{itemize}
Hence, we have our theorem. \hfill\(\Box\)

\vspace{5mm}

\textbf{Proof of Theorem \ref{a-lp-menability}:} 
Let $\phi_\alpha:\Gamma_\alpha\rightarrow G$ be a group homomorphism for a.e. $\alpha\in \mathcal{A}(\Omega, T^j)$, where $G$ is a a-$L^p$-menable group. Since the Property $F_{L^p}$  (i.e. having fixed point for any affine isometric action on $L^p$ space) passes to quotients (see Proposition 2.15 in \cite{BFGM07}), $\phi_\alpha(\Gamma_\alpha)$ has also Property 
$F_{L^p}$. On the other hand, a subgroup of an a-$L^p$-menable group, is a-$L^p$-menable. Therefore, $\phi_\alpha(\Gamma_\alpha)$ is an a-$L^p$-menable group. But, any group having both 
Property $F_{L^p}$ and a-$L^p$-menability is finite. Hence, we have our theorem.  \hfill\(\Box\)

\vspace{5mm}

\textbf{Proof of Corollary \ref{hsa-lp-menability}:} There is a general fact which follows from the proof of Theorem \ref{a-lp-menability}: Any group homomorphism from a group with Property $F_{L^p}$ to an a-$L^p$-menable group has finite image. Therefore, it suffices to prove that any finite index subgroup of $\Gamma_\alpha$ has Property $F_{L^p}$, which follows directly from Proposition 8.8 of \cite{BFGM07}.
Hence we have our corollary. \hfill\(\Box\)

\vspace{5mm}

\textbf{Proof of Theorem \ref{K-amenable}:} Let $\phi_\alpha:\Gamma_\alpha\rightarrow G$ be a group homomorphism for a.e. $\alpha\in \mathcal{A}(\Omega, T^j)$, where $G$ is a K-amenable group. Since Property (T) passes to quotient groups, using Theorem \ref{propt} we obtain that $\phi_\alpha(\Gamma_\alpha)$ has Property (T) for a.e. $\alpha$. On the other hand, $\phi_\alpha(\Gamma_\alpha)$ is K-amenable by Theorem \ref{kamensub}. Now, using Theorem \ref{kamen+t} we obtain that $\phi_\alpha(\Gamma_\alpha)$ is finite for a.e. $\alpha\in \mathcal{A}(\Omega, T^j)$.
\hfill\(\Box\)

\vspace{5mm}

\textbf{Proof of Corollary \ref{hsK-amenable}:} There is a general fact which follows from the proof of Theorem \ref{K-amenable}: Any group homomorphism from a group with Property (T) to a K-amenable group has finite image. Therefore, it suffices to prove that any finite index subgroup of $\Gamma_\alpha$ has Property (T), which follows directly from Proposition 2.5.7 of \cite{BHV08} (or Proposition 8.8 of \cite{BFGM07} replacing $p=2$). Hence we have our corollary. \hfill\(\Box\)

\section{Open questions}

Among linear groups, we know that Gromov's random monster group has super-rigidity with respect to  lattices in $SL_2(\mathbb{R}$, $SO(n,1)$, $SU(n,1)$, $Sp(n,1)$ and with respect to braid groups. But,
We do not know whether the super-rigidity phenomenon is exhibited for $SL_n(\mathbb{Z})$ ( $n\geq 3$) or for $Sp_{2n}(\mathbb{Z})$ ($n\geq 2$) .  We can pose this question in general framework as follows :

\vspace{5mm}

\textbf{Question 1:} Does $\Gamma_\alpha$ have super-rigidity with respect to $G$ for a. e. $\alpha\in\mathcal{A}(\Omega,T^j)$, where $G$ is a lattice in a product of higher rank simple Lie groups or higher rank simple algebraic groups over local fields ?

\vspace{5mm}

However, this super-rigidity phenomenon holds for $Out(F_N)$. But, we do not know to what extent this phenomenon extends to the outer automorphism group of a right-angled Artin group (RAAG) $A_\Lambda$, where $\Lambda$ is the defining graph. If $\Lambda$ is the graph with disjoint vertices, then $A_\Lambda$ is the free group $F_n$, and if $\Lambda$ is the complete graph, then $A_\Lambda$ is the free abelian group $\mathbb{Z}^n$. Moreover, $Out(\mathbb{Z}^n)=GL_n(\mathbb{Z})$. One therefore expects traits shared by $F_n$ to be shared by an arbitrary RAAG. Therefore, we have the following question:

\vspace{5mm}

\textbf{Question 2:} Does $\Gamma_\alpha$ have super-rigidity with respect to $Out(A_\Lambda)$ for a. e. $\alpha\in\mathcal{A}(\Omega,T^j)$?

\vspace{5mm}

We have seen that $\Gamma_\alpha$ has super-rigidity with respect to the following classes of biautomatic groups (for the definition see \cite{Eps92}):
\begin{itemize}
\item[(1)] hyperbolic groups; because of having a-$F_{L^p}$-menability $\Gamma_\alpha$ has super-rigidity with respect to them and for their biautomatic property we refer to \cite{Eps92});
\item[(2)] groups acting properly discontinuously and co-compactly on a $CAT(0)$ cubical complex; because of having Haagerup property $\Gamma_\alpha$ has super-rigidity with respect to them and for their biautomatic property we refer to \cite{Swi06};
\item[(3)] the central extension of hyperbolic groups; it follows from Theorem \ref{shortexact} that  $\Gamma_\alpha$ has super-rigidity with respect to them and for their biautomatic property (see \cite{NR97}) ;
\end{itemize}
But, we do not know whether it is true for all biautomatic groups.

\vspace{5mm}

\textbf{Question 4:} Does $\Gamma_\alpha$ have super-rigidity with respect to $G$ for a. e. $\alpha\in\mathcal{A}(\Omega,T^j)$, where $G$ is a biautomatic group?

\vspace{5mm}

We have seen that $\Gamma_\alpha$ has both super-rigidity and hereditary super-rigidity with respect to a-$FL^p$-menable group with $1< p<\infty$ and K-amenable group. But, we do not know the answer of the following question:

\vspace{5mm}

\textbf{Question 5:} Does $\Gamma_\alpha$ have hereditary super-rigidity with respect to $G$, where $G$ is a mapping class group or braid group or $Out(F_N)$ or $Aut(F_N)$ or hierarchically hyperbolic group ?

\section{Acknowledgements}

I would like to thank National Board for Higher Mathematics (NBHM), India for supporting this project financially. I would also like to thank John Mackay, Mahan Mj. and Romain Tessera for useful discussions.

\end{document}